\documentclass{amsart}
\usepackage{amsfonts, amsthm, amsmath}
\usepackage{graphicx}
\usepackage{amssymb}
\usepackage{epstopdf}

\newtheorem{Def}{Definition}
\newtheorem{thm}{Theorem}

\newtheorem{lem}{Lemma}

\newtheorem{cor}{Corollary}

\newcommand{\tr}{\vartriangle}
\newcommand{\angl}{\measuredangle}

\begin{document}

\title{On Some Results related to Napoleon's Configurations}
\author{Nikolay Dimitrov} 
\address{Department of Mathematics and Statistics, McGill
University, 805 Sherbrooke W., Montreal, QC H3A 2K6, Canada}
\email{dimitrov@math.mcgill.ca}




\begin{abstract}
The goal of this paper is to give a purely geometric proof of a
theorem by Branko Gr\"unbaum concerning configuration of triangles
coming from the classical Napoleon's theorem in planar Euclidean
geometry.
\end{abstract}

\maketitle

\section{Introduction}
In this short article I discuss some results from planar Euclidean
geometry which have a close connection to Napoleon's Theorem. They
are summarized in Theorem \ref{Main}. The statements of Theorem
\ref{Main} appear in \cite{G} where the proofs are based on
coordinate descriptions and algebraic computations. The proofs
given in the current paper use only purely geometric methods with
hardly any algebraic computations and are entirely in the spirit
of synthetic Euclidean geometry. Both Theorem \ref{Main} and
Napoleon's Theorem (see Theorem \ref{Napoleon}) are stated as
simple geometric results so it is appropriate to try and give
arguments that remain in the same simple geometric domain.
Moreover, in this way one can get a better  feeling about the
geometry and the properties of Napoleon's configurations.

\begin{Def} \label{Configuration}
Let $\tr ABC$ be an arbitrary triangle. We say that the points
$A_1, B_1$ and $C_1$ form a non-overlapping Napoleon's
configuration for the triangle $\tr ABC$ if all three triangles
$\tr ABC_1,$ $\tr AB_1C$ and $\tr A_1BC$ are equilateral and no
one of them overlaps with $\tr ABC$ (see Figure \ref{Config}.)
Alternatively, we say that the points $A_1^{'}, B_1^{'}$ and
$C_1^{'}$ form an overlapping Napoleon's configuration for $\tr
ABC$ if all three triangles $\tr ABC_1^{'},$ $\tr AB_1^{'}C$ and
$\tr A_1^{'}BC$ are equilateral and all of them overlap with $\tr
ABC$.
\end{Def}

\begin{figure}
\includegraphics[width=120mm]{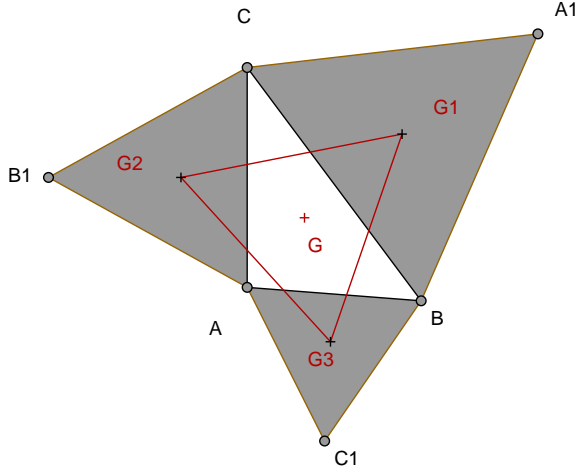}
\caption{A non-overlapping Napoleon's configuration and the first
part of Napoleon's Theorem.} \label{Config}
\end{figure}

\section{The Main Result}

The main result we are going to discuss is the following theorem:

\begin{thm} \label{Main}
Let us have an arbitrary triangle $\tr ABC$ and let $A_1, B_1$ and
$C_1$ form a non-overlapping Napoleon's configuration for that
triangle. Denote the midpoints of $B_1C_1, C_1A_1$ and $A_1B_1$ by
$A_2, B_2$ and $C_2$ respectively. Also, denote the centroids of
triangles $\tr A_1BC, \tr AB_1C$ and $\tr ABC_1$ by $G_1, G_2$ and
$G_3$ respectively. Then the following statements are true:
\vspace{2mm}

\noindent {\em 1.} The triangles $\tr A_2B_2C, \tr AB_2C_2$ and
$\tr A_2BC_2$ are equilateral; 
\vspace{2mm}

\noindent {\em 2.} The centroids $A^{*}, B^{*}, C^{*}$ of $\tr
AB_2C_2, \tr A_2B_2C, \tr A_2BC_2$ respectively are vertices of an
equilateral triangle, whose centroid coincides with the centroid
$G$ of $\tr ABC;$ \vspace{2mm}

Similarly, let $A_1^{'}, B_1^{'}$ and $C_1^{'}$ be an overlapping
Napoleon's configuration for $\tr ABC.$ Denote the midpoints of
$B_1^{'}C_1^{'}, C_1^{'}A_1^{'}$ and $A_1^{'}B_1^{'}$ by $A_2^{'},
B_2^{'}$ and $C_2^{'}$ respectively. Also, denote the centroids of
triangles $\tr A^{'}_1B^{'}C^{'}, \tr A^{'}B^{'}_1C^{'}$ and $\tr
A^{'}B^{'}C^{'}_1$ by $G^{'}_1, G^{'}_2$ and $G^{'}_3$
respectively.  Then \vspace{2mm}

\noindent {\em 3.} The triangles $\tr A_2^{'}B_2^{'}C, \tr
AB_2^{'}C_2^{'}$ and $\tr A_2^{'}BC_2^{'}$ are equilateral;
\vspace{2mm}

\noindent {\em 4.} The centroids $A^{**}, B^{**}, C^{**}$ of $\tr
AB_2^{'}C_2^{'}, \tr A_2^{'}BC_2^{'}, \tr A_2^{'}B_2^{'}C$
respectively are vertices of an equilateral triangle, whose
centroid coincides with the centroid $G$ of $\tr ABC;$
\vspace{2mm}

\noindent{\em 5.} Triangle $\tr A^{*}B^{*}C^{*}$ is homothetic to
 the triangle $\tr G^{'}_1G^{'}_2G^{'}_3$ with a homothetic center $G$
 and a coefficient of similarity $-1/2;$ \vspace{2mm}

\noindent{\em 6.} Triangle $\tr A^{**}B^{**}C^{**}$ is homothetic
to the triangle $\tr G_1G_2G_3$ with a homothetic center $G$ and a
coefficient of similarity $-1/2.$ \vspace{2mm}

 \noindent {\em 7.} The area of $\tr ABC$ equals
four times the algebraic sum of the areas of $\tr A^{*}B^{*}C^{*}$
and $\tr A^{**}B^{**}C^{**}.$
\end{thm}


\section{Napoleon's Theorem}

Napoleon's Theorem is a beautiful result from planar Euclidean
geometry. One can approach its proof in different ways. The proof
given here is a geometric one and in agreement with the
planimetric nature of the statement. Before we state and prove
Napoleon's Theorem we are going to need the following lemma:

\begin{lem} \label{BasicLemma}
Given an arbitrary $\tr ABC$, let $A_1, B_1$ and $C_1$ form a
non-overlapping Napoleon's configuration for that triangle. Then,
the following properties are true:
\smallskip

\noindent {\em 1.} The segments $AA_1, BB_1$ and $CC_1$ are of
equal length. In other words, $AA_1 = BB_1 = CC_1;$
\smallskip

\noindent {\em 2.} They intersect at a common point. Let us denote
it by $J;$
\smallskip

\noindent {\em 3.} $\angl AJB = \angl BJC = \angl CJA =
120^{\circ}.$
\smallskip

\noindent {\em 4.} The circles $K_1, K_2$ and $K_3$ circumscribed
around the equilateral triangles $\tr A_1BC, \tr AB_1C$ and $\tr
ABC_1$ respectively pass through the point $J$ (see Figure
\ref{FigLemma1}.)
\end{lem}

\begin{figure}
\centering%
\includegraphics[width=120mm]{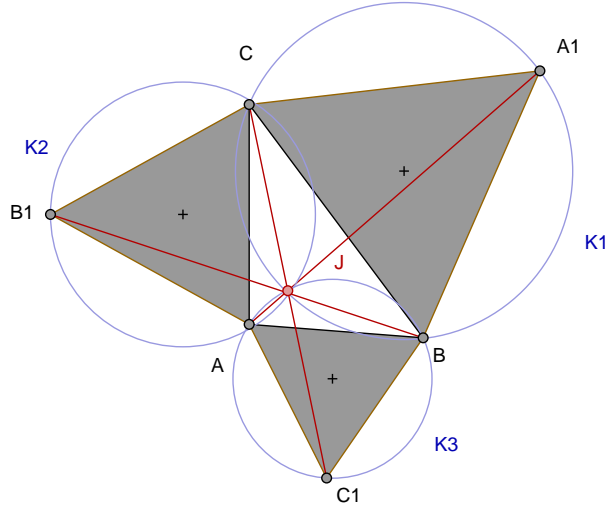}
\caption{Constructions in the Proof of Lemma \ref{BasicLemma}.}
\label{FigLemma1}
\end{figure}

\begin{proof}
Perform a $60^{\circ}$ rotation $R_A$ around the point $A$ in
counterclockwise direction. Since $AC = AB_1$ and $\angl CAB_1 =
60^{\circ},$ the point $C$ is mapped to the point $B_1.$
Similarly, $C_1$ is mapped to $B.$ Therefore the segment $CC_1$
maps to the segment $B_1B.$ This implies that $BB_1 = CC_1$ (see
Figure \ref{FigLemma1}.) Moreover, if we denote by $J$ the
intersection point of $BB_1$ and $CC_1,$ then $\angl CJB_1 = \angl
C_1JB = 60^{\circ}$ and $\angl BJC = 180^{\circ} - \angl CJB_1 =
180^{\circ} - 60^{\circ} = 120^{\circ}.$ We are going to show that
the points $A, J$ and $A_1$ lie on the same line.

Notice that $\angl CJB_1 = \angl CAB_1 = 60^{\circ}.$ Therefore
the quadrilateral $CB_1AJ$ is inscribed in a circle $K_2.$ Then,
$\angl CJA = 180^{\circ} - \angl AB_1C = 180^{\circ} - 60^{\circ}
= 120^{\circ}.$ Since $\angl BJC + \angl CA_1B = 120^{\circ} +
60^{\circ} = 180^{\circ},$ the points $B, A_1, C$ and $J$ lie on a
circle $K_1.$ From here we can conclude that $\angl A_1JC = \angl
A_1BC = 60^{\circ}.$ Then, $\angl A_1JA = \angl A_1JC + \angl CJA
= 60^{\circ} + 120^{\circ} = 180^{\circ}.$ That means that $J$
belongs to the straight line $AA_1.$

If we perform another $60^{\circ}$ counterclockwise rotation
$R_B,$ this time around the point $B,$ it will turn out that
$AA_1$ is mapped to $C_1C.$ Therefore, $AA_1 = CC_1.$ Also, $\angl
AJB = 360^{\circ} - \angl BJC - \angl CJA = 360^{\circ} -
120^{\circ} - 120^{\circ} = 120^{\circ}.$ Since $\angl AJB + \angl
BC_1A = 120^{\circ} + 60^{\circ} = 180^{\circ},$ the points $A,
C_1, B$ and $J$ lie on a circle $K_3.$ We see that the circles
$K_1, K_2, K_3$ all pass through the same point $J.$ This
completes the proof of Lemma \ref{BasicLemma}.
\end{proof}

Now we are ready to state and prove Napoleon's Theorem.

\begin{thm} \label{Napoleon} Let $\tr ABC$ be an arbitrary
triangle and let $G$ be its centroid. Then, the following
statements are true: \vspace{2mm}

\noindent {\em 1.} Assume $A_1, B_1$ and $C_1$ form a
non-overlapping Napoleon's configuration for that triangle. Denote
the centroids of triangles $\tr A_1BC, \tr AB_1C$ and $\tr ABC_1$
by $G_1, G_2$ and $G_3$ respectively. Then, the triangle $\tr
G_1G_2G_3$ is equilateral with a centroid coinciding with the
point $G;$
\vspace{2mm}

\noindent {\em 2.} Let $A^{'}_1, B^{'}_1$ and $C^{'}_1$ form an
overlapping Napoleon's configuration for that triangle. Denote the
centroids of triangles $\tr A^{'}_1B^{'}C^{'}, \tr
A^{'}B^{'}_1C^{'}$ and $\tr A^{'}B^{'}C^{'}_1$ by $G^{'}_1,
G^{'}_2$ and $G^{'}_3$ respectively. Then, the triangle $\tr
G^{'}_1G^{'}_2G^{'}_3$ is equilateral with a centroid coinciding
with the point $G;$ \vspace{2mm}

\noindent {\em 3.} The area of $\tr ABC$ equals the algebraic sum
of the areas of $\tr G_1G_2G_3$  and $\tr G^{'}_1G^{'}_2G^{'}_3.$
\end{thm}

\begin{figure}
\centering%
\includegraphics[width=120mm]{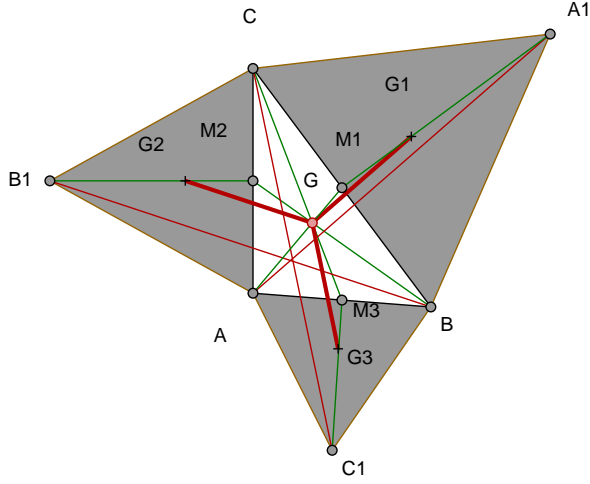}
\caption{Constructions in the Proof of Napoleon's Theorem.}
\label{FigProofNapoleon1}
\end{figure}

\begin{proof}
We start with first claim of the theorem. Let $M_1, M_2$ and $M_3$
be the midpoints of the edges $BC, CA$ and $AB$ respectively.
Since $G$ is the centroid of $\tr ABC$ and $G_2$ is the centroid
of $\tr AB_1C,$ we have the ratios $M_2G : M_2B = M_2G_2 : M_2B_1
= 1 : 3.$ Therefore, by the Intercept Theorem $GG_2 = \frac{1}{3}
BB_1$ and $GG_2$ is parallel to $BB_1.$ Analogously, $GG_1 =
\frac{1}{3} AA_1, GG_1$ is parallel to $AA_1,$ $GG_3 = \frac{1}{3}
CC_1$ and $GG_3$ is parallel to $CC_1.$ By part 1 of Lemma
\ref{BasicLemma}, $AA_1 = BB_1 = CC_1,$ hence $GG_1 = GG_2 =
GG_3.$ By part 3 of Lemma \ref{BasicLemma}, $\angl AJB = \angl BJC
= \angl CJA = 120^{\circ},$ so $\angl G_1GG_2 = \angl G_2GG_3 =
\angl G_3GG_1 = 120^{\circ}.$ We can conclude form here that $\tr
G_1G_2G \cong \tr G_2G_3G \cong \tr G_3G_1G$ and hence $G_1G_2 =
G_2G_3 = G_3G_1,$ that is the triangle $\tr G_1G_2G_3$ is
equilateral.

\begin{figure}
\centering%
\includegraphics[width=120mm]{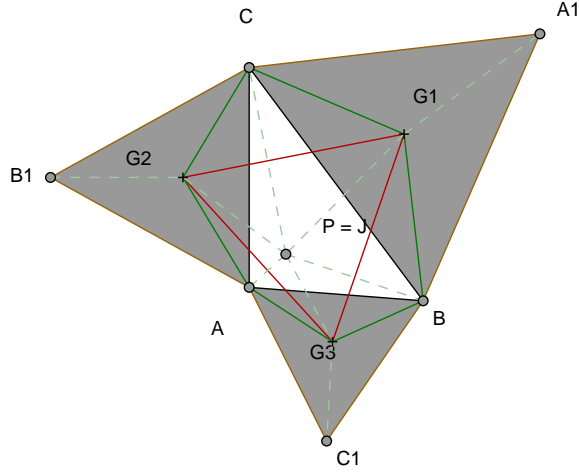}
\caption{Constructions in the Proof of Napoleon's Theorem.}
\label{FigProofNapoleon2}
\end{figure}

The proof of claim 2 from Napoleon's Theorem is analogous to the
proof of claim 1. We just have to consider overlapping
configurations and to rename the notations appropriately.

In order to prove claim 3 form Theorem \ref{Napoleon}, we are
going to show that $Area(\tr G_1G_2G_3) = \frac{1}{2} Area(\tr
ABC) + \frac{1}{6}(Area(\tr A_1BC) + Area(\tr AB_1C) + Area(\tr
ABC_1)).$ Let point $P$ be the reflection image of the vertex $C$
with respect to the line $G_1G_2.$ In other words, $P$ is chosen
so that $G_1G_2$ is the perpendicular bisector of $CP.$ Hence,
$\tr G_1G_2C \cong \tr G_1G_2P$ and $G_2P = G_2C = G_2A.$ If we
denote $\angl G_1G_2C = \alpha$ then $\angl PG_2G_1 = \alpha.$ On
the one hand, $\angl AG_2P = \angl AG_2C - \angl PG_2C =
120^{\circ} - \angl PG_2C = 120^{\circ} - (\angl PG_2G_1 + \angl
G_1G_2C) = 120^{\circ} - 2 \alpha.$ On the other hand, $\angl
G_3G_2P = \angl G_3G_2G_1 - \angl PG_2G_1 = 60^{\circ} - \alpha.$
Therefore, $\angl AG_2G_3 = \angl AG_2P - \angl G_3G_2P =
120^{\circ} - 2 \alpha - (60^{\circ} - \alpha) = 60^{\circ} -
\alpha.$ Since $G_2P = G_2A$ and $\angl AG_2G_3 = \angl G_3G_2P =
60^{\circ} - \alpha,$ the line $G_2G_3$ is the bisector of $\angl
AG_2P$ in the isosceles triangle $\tr AG_2P,$ and hence it is the
perpendicular bisector of the segment $AP.$ Therefore, $P$ is the
reflection image of $A$ with respect to $G_2G_3$ and $\tr G_2G_3A
\cong \tr G_2G_3P.$ Analogously, we can show that the reflection
of $B$ with respect to $G_3G_1$ is again $P$ and $\tr G_3G_1B
\cong \tr G_3G_1P.$ All of the arguments above lead to the
conclusion that $Area(\tr G_1G_2G_3) = Area(\tr G_1G_2P) +
Area(\tr G_2G_3P) + Area(\tr G_3G_1P) = Area(\tr G_1G_2C) +
Area(\tr G_2G_3A) + Area(\tr G_3G_1B),$ so $Area(\tr G_1G_2G_3) =
\frac{1}{2} Area(AG_3BG_1CG_2).$ Notice that $Area(AG_3BG_1CG_2) =
Area(\tr ABC) + Area(\tr AG_3B) + Area(\tr BG_1C) + Area(\tr
CB_2A) = Area(\tr ABC) + \frac{1}{3}(Area(\tr A_1BC) + Area(\tr
AB_1C) + Area(\tr ABC_1)).$ It follows from here that $Area(\tr
G_1G_2G_3) = \frac{1}{2} Area(\tr ABC) + \frac{1}{6}(Area(\tr
A_1BC) + Area(\tr AB_1C) + Area(\tr ABC_1)).$

Using analogous arguments, one can show that $Area(\tr
G_1^{'}G_2^{'}G_3^{'}) = \frac{1}{2} Area(\tr ABC) -
\frac{1}{6}(Area(\tr A_1^{'}BC) + Area(\tr AB_1^{'}C) + Area(\tr
ABC_1^{'})).$ Now, we can deduce that $Area(\tr G_1G_2G_3) +
Area(\tr G_1^{'}G_2^{'}G_3^{'}) = Area(\tr ABC)$

An additional observation is that $G_1P = G_1B = G_1C = G_1A_1$
and therefore $P$ lies on the circle $K_1,$ circumscribed around
$\tr A_1BC$ (see Lemma \ref{BasicLemma} and Figure
\ref{FigLemma1}.) Similarly, $P$ lies on the circles $K_2$ and
$K_3$ circumscribed around $\tr AB_1C$ and $\tr ABC_1$
respectively. That implies $P$ is the intersection point of $K_1,
K_2$ and $K_3,$ which was already denoted by $J,$ i.e. $P \equiv
J.$
\end{proof}

\begin{figure}
\centering%
\includegraphics[width=120mm]{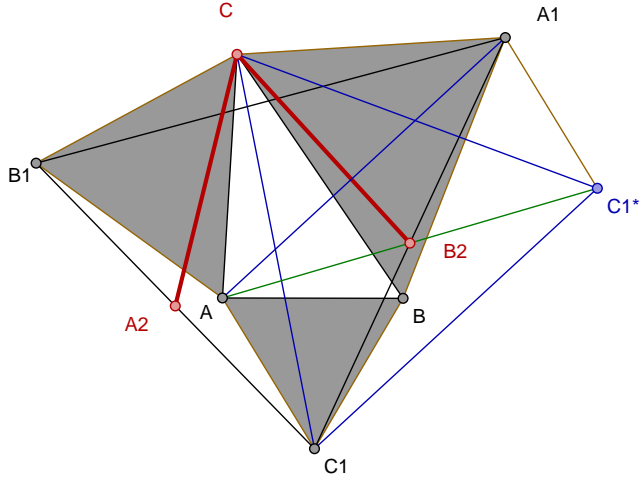}
\caption{Constructions in the Proof of Lemma \ref{Lemma1}.}
\label{FigLemma2}
\end{figure}

\section{Proof of Theorem \ref{Main}}

This section contains the proof of the main result, namely Theorem
\ref{Main}. To prove this statement we are going to use several
lemmas and corollaries which together will give us the desired
result.

The next lemma is essentially the proof of fact 1 from Theorem
\ref{Main}.

\begin{lem} \label{Lemma1}
In the setting of Theorem \ref{Main}, the points $A_2, B_2$ and
$C$ form an equilateral triangle (see Figure \ref{FigLemma2}.)
\end{lem}

\begin{figure}
\centering%
\includegraphics[width=120mm]{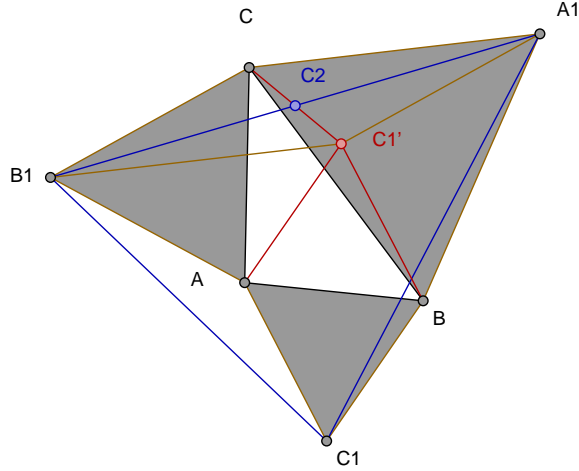}
\caption{Constructions in the Proof of Lemma \ref{Midpoint}.}
\label{FigMidpoint}
\end{figure}
\begin{proof}
Consider a $60^{\circ}$ rotation $R_C$ around the point $C$ in
counterclockwise direction. The point $B_1$ maps to $A.$ Denote by
$C_1^{*}$ the image of the point $C_1$ (Figure \ref{FigLemma2}.)
Then, $B_1C_1$ maps to $AC_1^{*}.$ We are going to show that the
point $B_2$ is the image of $A_2$ under the rotation $R_C.$ Since
the midpoint $A_2$ of $B_1C_1$ maps to the midpoint of the image
$AC_1^{*},$ we need to prove that $B_2$ lies on $AC_1^{*}$ and is
the midpoint of that segment.

By the properties of the rotation $R_C,$ we have that
$CC_1=CC_1^{*}$ and $\angl C_1CC_1^{*} = 60^{\circ}.$ Therefore
triangle $\tr CC_1C_1^{*}$ is equilateral and so by Lemma
\ref{BasicLemma} we can deduce that $C_1C_1^{*} = CC_1 = AA_1.$

Notice that the point $A_1$ is the image of $B$ under the rotation
$R_C.$ Since $C_1$ maps to $C_1^{*}$ then $BC_1$ maps to
$A_1C_1^{*}.$ Thus, $A_1C_1^{*} = BC_1 = AC_1.$

The facts that $CC_1 = AA_1$ and $A_1C_1^{*} = AC_1$ imply that
the quadrilateral $AA_1C_1^{*}C_1$ is a parallelogram. For any
parallelogram, the intersection point of the diagonals is the
midpoint for for both diagonals. That means that the midpoint
$B_2$ of the diagonal $C_1A_1$ lies on the diagonal $AC_1^{*}$ and
is the midpoint of $AC_1^{*}$. Therefore, $B_2$ is the image of
$A_2$ under the rotation $R_C.$ Hence, $CA_2 = CB_2$ and $\angl
A_2CB_2 = 60^{\circ},$ i.e. the triangle $\tr A_2B_2C$ is
equilateral.
\end{proof}

We are going to need the following intermediate statement:

\begin{lem}\label{Midpoint}
Consider the equilateral triangle $\tr ABC^{'}_1,$ overlapping
$\tr ABC.$ Then, the midpoint $C_2$ of the segment $A_1B_1$ is
also the midpoint of $CC_1^{'}$ (see Figure \ref{FigMidpoint}.)
\end{lem}

\begin{proof}
Consider a $60^{\circ}$ degree clockwise rotation around the point
$A.$ Then $B$ maps to $C_1^{'}$ and $C$ maps to $B_1$. Therefore
the segment $BC$ maps to the segment $C_1^{'}B_1,$ so $BC =
C_1^{'}B_1.$

Now consider a $60^{\circ}$ degree counter-clockwise rotation
around the point $B.$ In this case $A$ maps to $C_1^{'}$ and $C$
maps to $A_1.$ Thus, the segment $AC$ maps to $C_1^{'}A_1,$ so $AC
= C_1^{'}A_1.$

From the two identities $BC = C_1^{'}B_1$ and $AC = C_1^{'}A_1$ it
can be concluded that the quadrilateral $B_1C_1^{'}A_1C$ is a
parallelogram. Therefore, the midpoint $C_2$ of the diagonal
$A_1B_1$ is also the midpoint of the diagonal $CC_1^{'}.$
\end{proof}

\begin{figure}
\centering%
\includegraphics[width=120mm]{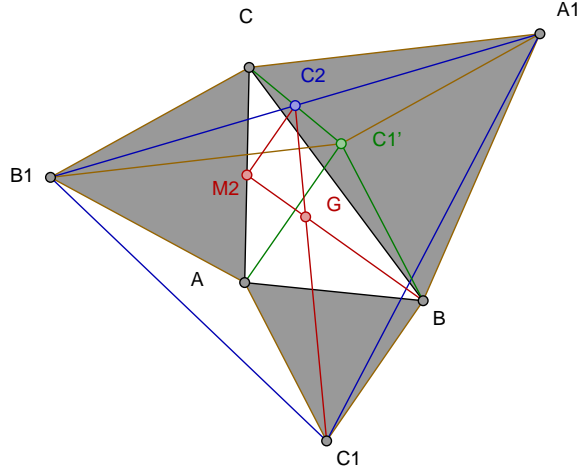}
\caption{Constructions in the Proof of Lemma
\ref{CommonCentroid}.} \label{FigCommonCentroid}
\end{figure}

Next, we are going to locate the centroids of $\tr A_1B_1C_1$ and
$\tr A_2B_2C_2.$

\begin{lem}\label{CommonCentroid}
The centroids of $\tr A_1B_1C_1$ and $\tr A_2B_2C_2$ coincide with
the centroid $G$ of $\tr ABC$ (see Figure
\ref{FigCommonCentroid}.)
\end{lem}

\begin{proof}
 Let $M_2$ be the midpoint of $AC.$ Then $M_2C_2$ is a mid-segment of the triangle
$\tr AC_1^{'}C.$ Therefore, $M_2C_2$ is parallel to $AC_1^{'}$ and
$2 M_2C_2 = AC_1^{'}.$ Since triangles $\tr ABC_1$ and $\tr
ABC_1^{'}$ are equilateral, the quadrilateral $AC_1BC_1^{'}$ is a
rhombus, so $AC_1^{'} = C_1B$ and $AC_1^{'}$ is parallel to
$C_1B.$ Hence $M_2C_2$ is parallel to $ C_1B$ and $2 M_2C_2 =
C_1B.$ Let $G^{'}$ be the intersection point of $BM_2$ and
$C_1C_2.$ From here we can deduce that $BG : G^{'}M_2 = C_1G^{'} :
G^{'}C_2 = BC_1 : C_2M_2 = 2 : 1.$ But for the centroid $G$ of
$\tr ABC$ it is true that $BG : GM_2 = 2 : 1,$ so $G \equiv G^{'}$
and $G$ is the centroid of $\tr A_1B_1C_1.$ Since the triangles
$\tr A_1B_1C_1$ and $\tr A_2B_2C_2$ have a common centroid, the
statement is proved.
\end{proof}

The following corollary proves statements 1 and 2 from Theorem
\ref{Main}.

\begin{cor} \label{ProofConfig}
The points $A, B$ and $C$ form an overlapping Napoleon's
configuration for $\tr A_2B_2C_2.$ Moreover, the centroids $A^{*},
B^{*}, C^{*}$ of the equilateral triangles $\tr AB_2C_2, \tr
A_2BC_2$ and $\tr A_2B_2C$ respectively form an equilateral
triangle, whose centroid coincides with the centroid $G$ of $\tr
ABC.$
\end{cor}

\begin{proof}
By Lemma \ref{Lemma1} first applied to the triple $A,B_2,C_2,$
then to the triple $A_2, B, C_2$ and finally to the triple $A_2,
B_2, C$ we obtain the first statement of Corollary
\ref{ProofConfig}. Thus, the points $A, B$ and $C$ form an
overlapping Napoleon's configuration for $\tr A_2B_2C_2.$ By the
classical Napoleon's Theorem for overlapping configurations, it
follows that the centroids $A^{*}, B^{*}, C^{*}$ of $\tr AB_2C_2,
\tr A_2BC_2$ and $\tr A_2B_2C$ respectively form an equilateral
triangle whose centroid coincides with the centroid of $\tr
A_2B_2C_2.$ By Lemma \ref{CommonCentroid}, the centroid of $\tr
A_2B_2C_2$ coincides with the centroid $G$ of $\tr ABC.$ The
corollary is proved.
\end{proof}

Notice that the proof of statements 3 and 4 from Theorem
\ref{Main} is absolutely analogous to the proof of statements 1
and 2. All we have to do is to follow more or less the same
arguments, just changing the notation appropriately. What is left
is the verification of the last three claims from Theorem
\ref{Main}. We proceed with the following lemma:

\begin{figure}
\centering%
\includegraphics[width=120mm]{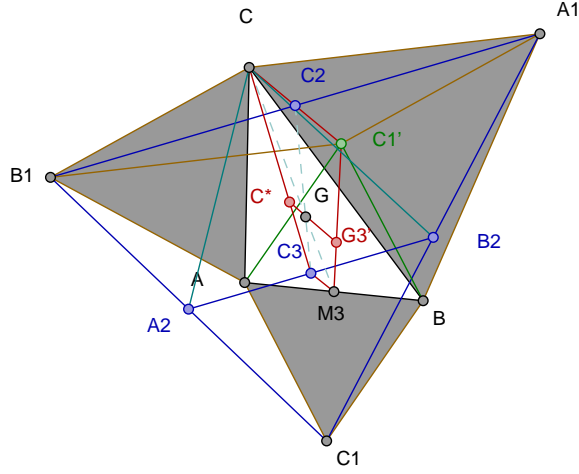}
\caption{Constructions in the Proof of Lemma \ref{Homothetic}.}
\label{FigHomothetic}
\end{figure}

\begin{lem} \label{Homothetic}
Consider the centroids $C^{*}$ and $G_3^{'}$ of the equilateral
triangles $\tr A_2B_2C$ and $\tr ABC_1^{'}$ respectively. Then
$G_3^{'}$ maps to $C^{*}$ under a homothetic transformation of
dilation factor $-1/2$ with respect to the centroid $G$ of $\tr
ABC$ (see Figure \ref{FigHomothetic}.)
\end{lem}

\begin{proof}
Perform a homothetic transformation of dilation factor $-1/2$ with
respect to the centroid $G$ of $\tr ABC.$ By Lemma
\ref{CommonCentroid} the point $G$ is also the centroid of $\tr
A_1B_1C_1.$ Then $A_1B_1$ maps to $A_2B_2$ and so the midpoint
$C_2$ of $A_1B_1$ maps to the midpoint $C_3$ of $A_2B_2.$ Also,
the vertex $C$ maps to the midpoint $M_3$ of $AB$ because $G$ is
the centroid
of $\tr ABC$ (see Figure \ref{FigHomothetic}.) 
From here we can conclude that $C_3G : GC_2 = 1 : 2$ and $M_3G :
GC = 1 : 2$ which transforms into $C_3G : C_3C_2 = 1 : 3$ and
$M_3G : M_3C = 1 : 3.$ As $C^{*}$ is the centroid of $\tr
A_2B_2C,$ we can see that $C_3C^{*} : C_3C = C_3 : C_3C_2 = C^{*}G
: CC_2 = 1 : 3$ and $C^{*}G$ is parallel to $CC_2.$ Similarly,
$G_3^{'}$ is the centroid of $\tr ABC_1^{'},$ so $M_3G_3^{'} :
M_3C_1^{'} = M_3G : M_3C = G_3^{'}G : C_1^{'}C = 1 : 3$ and
$G_3^{'}G$ is parallel to $C_1^{'}C.$ By Lemma \ref{Midpoint}
$C_2$ is the midpoint of $CC_1^{'}$ which means that both $GC^{*}$
and $GG_3^{'}$ are parallel to the same line $CC_2.$ Therefore $G$
belongs to $C^{*}G_3^{'}.$ Moreover, $C^{*}G = \frac{1}{3} CC_2=
\frac{1}{6} CC_1^{'}$ and $G_3^{'}G = \frac{1}{3} CC_1^{'}.$ Hence
$C^{*}G : GG_3^{'} = 1 : 2,$ so the point $C^{*}$ is the image of
the point $G_3^{'}$ under the homothetic transformation of factor
$-1/2$ with respect to $G.$
\end{proof}

After establishing the previous result, we are ready to claim
statements 5, 6 and 7 from Theorem \ref{Main}.

\begin{cor}\label{LastTwoStatements}
In the setting of Theorem \ref{Main}, triangle $\tr
A^{*}B^{*}C^{*}$ is homothetic to the triangle $\tr
G^{'}_1G^{'}_2G^{'}_3$ with a homothetic center $G$ and a
coefficient of similarity $-1/2.$ Similarly, triangle $\tr
A^{**}B^{**}C^{**}$ is homothetic to the triangle $\tr G_1G_2G_3$
with a homothetic center $G$ and a coefficient of similarity
$-1/2.$ Moreover, the area of $\tr ABC$ equals four times the
algebraic sum of the areas of $\tr A^{*}B^{*}C^{*}$ and $\tr
A^{**}B^{**}C^{**}.$
\end{cor}

\begin{proof}
Applying Lemma \ref{Homothetic} first to the pair of centroids
$C^{*}$ and $G_3^{'}$ of the equilateral triangles $\tr A_2B_2C$
and $\tr ABC_1^{'},$ then to the centroids $A^{*}$ and $G_1^{'}$
of the equilateral triangles $\tr AB_2C_2$ and $\tr A_1^{'}BC,$
and finally to the centroids $B^{*}$ and $G_2^{'}$ of the
equilateral triangles $\tr A_2BC_2$ and $\tr AB_1^{'}C,$ we
conclude that triangle $\tr A^{*}B^{*}C^{*}$ is homothetic to the
triangle $\tr G^{'}_1G^{'}_2G^{'}_3$ with respect to $G$ and a
dilation coefficient $-1/2.$ Analogously, the same is true for the
equilateral triangles $\tr A^{**}B^{**}C^{**}$ and $\tr
G_1G_2G_3.$ Finally, due to the homothety, the area of $\tr
A^{*}B^{*}C^{*}$ is $1/4$ of the area of $\tr
G^{'}_1G^{'}_2G^{'}_3$ and the area of $\tr A^{**}B^{**}C^{**}$ is
$1/4$ of the area of $\tr G_1G_2G_3.$ Since by Napoleon's Theorem
the area of $\tr ABC$ equals the algebraic sum of the areas of
$\tr G_1G_2G_3$ and $\tr G^{'}_1G^{'}_2G^{'}_3,$ we conclude that
the area of $\tr ABC$ equals four times the algebraic sum of the
areas of $\tr A^{*}B^{*}C^{*}$ and $\tr A^{**}B^{**}C^{**}.$ This
completes the proof of the corollary.
\end{proof}


\begin{thebibliography}{00}

\bibitem{G} Branko Gr\"unbaum,
{\em A Relative of "Napoleon's Theorem"} Geombinatorics, 10, pp.
116 - 121, Boston/Basel/Berlin, 2001.

\end{thebibliography}
\end{document}